\documentclass{amsart}
\usepackage{amssymb,latexsym}
\usepackage{amsmath,amsfonts,amsthm}
\usepackage{verbatim}
\theoremstyle{plain}
\newtheorem{theorem}{Theorem}[section]

\newtheorem{lemma}[theorem]{Lemma}
\newtheorem{cor}[theorem]{Corollary}
\theoremstyle{definition}
\newtheorem{definition}[theorem]{Definition}
\theoremstyle{remark}
\newtheorem{remark}[theorem]{Remark}

\def\co{\colon\thinspace}

\def\barnu{\bar{\nu}}

\begin{document}
\title{Spectral numbers in Floer theories}\author{Michael Usher}\email{musher@math.princeton.edu}\address{Princeton University\\Fine Hall\\Washington Road\\Princeton, NJ 08544} 
\begin{abstract}
The chain complexes underlying Floer homology theories typically carry a real-valued filtration, allowing one to associate to each Floer homology class a \emph{spectral number} defined as the infimum of the filtration levels of chains representing that class.  These spectral numbers have been studied extensively in the case of Hamiltonian Floer homology by Oh, Schwarz, and others.  We prove that the spectral number associated to any nonzero Floer homology class is always finite, and that the infimum in the definition of the spectral number is always attained.  In the Hamiltonian case, this implies that what is known as the ``nondegenerate spectrality'' axiom holds on all closed symplectic manifolds.  Our proofs are entirely algebraic  and apply to any Floer-type theory (including Novikov homology) satisfying certain standard formal properties.   The key ingredient is a theorem about the existence of best approximations of arbitrary elements of finitely generated free modules over Novikov rings by elements of prescribed submodules with respect to a certain family of non-Archimedean metrics. 
\end{abstract}
\maketitle
\section{Introduction}

In the various guises of Floer homology, or indeed its forerunners Morse homology and Novikov homology, one obtains a chain complex $C_*$ from the critical points of an ``action functional'' $\mathcal{A}$ on some configuration space, with boundary operator obtained from an enumeration of certain objects that are interpreted as negative gradient flowlines of $\mathcal{A}$.  There is generally some set of allowable perturbations of $\mathcal{A}$, with any two choices of perturbation giving rise to canonically isomorphic homology groups $H_*$. However, the nature of the construction implies that the chain complex also carries a natural \emph{filtration} by $\mathbb{R}$, with the subcomplex $C_{*}^{\lambda}$ generated by the critical points having action at most $\lambda$; typically the homologies $H_{*}^{\lambda}$ of these filtered groups are \emph{not} independent of the way in which $\mathcal{A}$ is perturbed.  Now in any chain complex carrying a filtration by $\mathbb{R}$, to each homology class $a$ of the complex one can associate a \emph{spectral number}, defined as the infimum of all $\lambda$ with the property that $a$ lies in the image of the inclusion-induced map $H_{*}^{\lambda}\to H_*$.  These spectral numbers have been studied in some detail in the case of \emph{Hamiltonian Floer homology} (see \cite{Oh} for a survey); in this case the allowed perturbations of the action functional correspond to Hamiltonian flows on a symplectic manifold, and the properties of the spectral numbers have yielded interesting information about Hamiltonian dynamics.  

The work described in the present note was motivated by the work of Oh, et al., concerning spectral numbers in Hamiltonian Floer homology.  The result of greatest interest to the Hamiltonian case is that what is known as the ``nondegenerate spectrality'' axiom holds on general closed symplectic manifolds.  This result  is also proven in  \cite{Oh3} in the case that the manifold is strongly semipositive.  In addition to not depending on semipositivity, our proof is conceptually quite different, and a good deal shorter.  As explained in Section 6.1 of \cite{Oh},  the nondegenerate spectrality axiom implies that the spectral number of a Hamiltonian flow is unchanged when the corresponding path of Hamiltonian diffeomorphisms is homotoped rel endpoints, and thus gives rise not just to a function defined on Hamiltonian functions but also to a function defined on the universal cover of the Hamiltonian diffeomorphism group.  With this fact in hand, many of the results about Hamiltonian dynamics that were proven using the spectral numbers for special classes of symplectic manifolds now become accessible for general symplectic manifolds.  (For instance, one can verify that the proofs in \cite{EP}, when combined with the nondegenerate spectrality axiom, go through to show that there exists a ``partial symplectic quasi-state'' on any closed symplectic manifold $(M,\omega)$.  As is shown in the proof of that paper's Theorem 2.1, this has the striking consequence that, if $\{F_1,\ldots,F_m\}$ is any finite collection of mutually Poisson-commuting functions on $M$, then there is some $(x_1,\ldots,x_m)\in\mathbb{R}^{m}$ such that the set $\cap_{i=1}^{m}f_{i}^{-1}(\{x_i\})$ cannot be displaced from itself by a Hamiltonian isotopy). 

However, the principal ingredient for the results in this paper is an algebraic result that is insensitive to the particular flavor of Floer homology under consideration.  Accordingly we obtain results applicable to many different theories, among which we mention in particular the fact that the spectral number is always nontrivial (\emph{i.e.}, not equal to $-\infty$) for any nonzero Floer homology class.  This suggests that these numbers may be worthy of study in contexts other than Hamiltonian Floer homology.  

In order to formulate our results in general terms, we now give a purely algebraic description of the context in which the results will apply.

\begin{definition} A \textbf{filtered Floer-Novikov complex} $\mathfrak{c}$  over a ring $R$ consists of the following data:\begin{enumerate}
\item[(1)] A principal $\Gamma$-bundle (with the discrete topology) \begin{align*} \Gamma\circlearrowright & P \\ & \downarrow \\ & S\end{align*}
where \begin{itemize} \item[(i)] $S$ is a finite set, and \item[(ii)] $\Gamma$ is a finitely generated  abelian group, written multiplicatively;\end{itemize}
\item[(2)] An ``action functional'' $\mathcal{A}\co P\to \mathbb{R}$ and a ``period homomorphism'' $\omega\co \Gamma\to \mathbb{R}$ satisfying \[ \mathcal{A}(g\cdot p)=\mathcal{A}(p)-\omega(g) \quad (g\in \Gamma,p\in P).\]
\item[(3)] A map $n\co P\times P\to R$ satisfying the following conditions: 
\begin{itemize}\item[(i)] $n(p,p')=0$ unless  $\mathcal{A}(p)>\mathcal{A}(p')$
\item[(ii)] $n(g\cdot p,g\cdot p')=n(p,p')$ for all $p,p'\in P,g\in \Gamma$
\item[(iii)] For each $p\in P$, the formal sum \[ \partial p=\sum_{q\in P}n(p,q)q\] belongs to the \emph{Floer chain complex} \[ C_*(\mathfrak{c}):=\big\{\sum_{q\in P}a_qq|a_q\in R,(\forall C\in\mathbb{R})(\#\{q|a_q\neq 0,\mathcal{A}(q)>C\}<\infty) \big\}\]
\item[(iv)] Where the \emph{Novikov ring} of $\omega\co\Gamma\to\mathbb{R}$ is defined by \[ \Lambda_{\Gamma,\omega}=\big\{\sum_{g\in \Gamma}b_gg|b_g\in R,(\forall C\in\mathbb{R})(\#\{g|b_g\neq 0,\omega(g)<C\}<\infty)\big\}\] 
and where $C_*$ inherits the structure of a $\Lambda_{\Gamma,\omega}$-module in the obvious way from the $\Gamma$-action on $P$, the operator $\partial\co P\to C_*$ defined in (iii) extends to a 
$\Lambda_{\Gamma,\omega}$-module homomorphism  \[ \partial\co C_*\to C_* \mbox{  which moreover satisfies } \partial\circ\partial=0.\]
\end{itemize}
\end{enumerate}
\end{definition}

Note that we use a ``downward completion'' with respect to $\mathcal{A}$ to define the Floer chain complex but an upward completion with respect to $\omega$ to define the Novikov ring; this is consistent with the minus sign that appears in (2) above.  

If $\lambda\in\mathbb{R}$, define  \[ C^{\lambda}_{*}(\mathfrak{c})=\big\{\sum_{q\in P:\mathcal{A}(q)\leq\lambda}a_qq|a_q\in R,(\forall C\in\mathbb{R})(\#\{q|a_q\neq 0,\mathcal{A}(q)>C\}<\infty) \big\} \] 
The condition 3(i) in the definition of a  filtered Floer-Novikov complex implies that the boundary operator $\partial$ restricts to maps \[ \partial\co C^{\lambda}_{*}(\mathfrak{c})\to C^{\lambda}_{*}(\mathfrak{c}).\]

So set \begin{align*} 
H_{*}(\mathfrak{c})=\frac{\ker(\partial \co C_*(\mathfrak{c})\to C_*(\mathfrak{c}))}{Im(\partial \co C_*(\mathfrak{c})\to C_*(\mathfrak{c}))},
\quad 
H_{*}^{\lambda}(\mathfrak{c})=\frac{\ker(\partial \co C_{*}^{\lambda}(\mathfrak{c})\to C_{*}^{\lambda}(\mathfrak{c}))}{Im(\partial \co C_{*}^{\lambda}(\mathfrak{c})\to C_{*}^{\lambda}(\mathfrak{c}))}.
\end{align*}

 We then have maps $\iota_*\co H_{*}^{\lambda}(\mathfrak{c})\to H_{*}(\mathfrak{c})$ induced by the inclusion of $C_{*}^{\lambda}(\mathfrak{c})$.  

\begin{definition} If $\mathfrak{c}$ is a   filtered Floer-Novikov complex with the notation as above, and if $a\in H_*(\mathfrak{c})$, the \emph{spectral number} of $a$ is \[ \rho(a)=\inf\{\lambda\in\mathbb{R}|a\in Im(\iota_*\co H_{*}^{\lambda}(\mathfrak{c})\to H_{*}(\mathfrak{c}))\}.\]
\end{definition}

For any nonzero $c=\sum_{p\in P}c_pp\in C_{*}(\mathfrak{c})$, the set $\{\mathcal{A}(p)|c_p\neq 0\}$ is bounded above, nonempty, and discrete and hence contains its supremum, which we denote by $\ell(c)$ (if $c=0$, put $\ell(c)=-\infty$).  We have thus defined a function \[ \ell\co C_*(\mathfrak{c})\to\mathbb{R}\cup\{-\infty\} \] such that \[ C_{*}^{\lambda}(\mathfrak{c})=\{c\in C_{*}(\mathfrak{c})|\ell(c)\leq\lambda\}.\]  An equivalent definition of the spectral number is then \[ \rho(a)=\inf\{\ell(c)|c\in C_*(\mathfrak{c}),[c]=a\}\] where $[c]$ denotes the homology class of $c$.

We can now state our main results.

\begin{theorem} \label{nontriv}  Let $\mathfrak{c}$ be a filtered Floer-Novikov complex over a Noetherian ring $R$.  Then for any $a\in H_*(\mathfrak{c})$ such that $a\neq 0$, we have $\rho(a)>-\infty$.  Further, there is $M\in\mathbb{R}$ such that for any $c\in C_*(\mathfrak{c})$ with $[c]=0$ there is $h\in C_*(\mathfrak{c})$ with $\partial h=c$ and $\ell(h)\leq \ell(c)+M$.\end{theorem}

\begin{theorem}  \label{tight} Let $\mathfrak{c}$ be a  filtered Floer-Novikov complex over a Noetherian ring $R$.   Then for every $a\in H_*(\mathfrak{c})$ there is $\alpha\in C_*(\mathfrak{c})$ such that $[\alpha]=a$ and \[ \ell(\alpha)=\rho(a).\] \end{theorem}

Recall that $R$ is the ring in which the $n(p,q)$ reside; in every filtered Floer-Novikov complex in the literature of which the author is aware $R$ is taken to be either a subfield of $\mathbb{C}$ or a quotient ring of $\mathbb{Z}$, so the assumption that $R$ is Noetherian is certainly a modest one.   The role of this assumption in the proof is that it guarantees that a certain submodule of a finitely generated module over the group ring $R[\ker\omega]$ will be finitely generated.  As we see in Remark \ref{non-noeth}, one can construct examples where our main theorems fail when the ring $K$ of Theorem \ref{alg} (which is set equal to $R[\ker \omega]$ in the proofs of Theorems \ref{nontriv} and \ref{tight}) is not Noetherian.

Note that the existence of the constant $M$ in Theorem \ref{nontriv} is reminiscent of Proposition A.4.9 of \cite{FOOO}.

Let us now recall two of the topological contexts in which graded filtered Floer-Novikov complexes arise.

\subsection{Novikov homology}  Let $M$ be a smooth closed manifold, and let $\theta$ be a closed 1-form on $M$ whose graph (as a submanifold of $T^*M$) is transverse to the zero section.  Let $\pi\co \tilde{M}\to M$ be any covering space of $M$ with the properties that the deck transformation group $\Gamma$ is abelian and $\pi^*\theta$ is exact; say $\pi^*\theta=d\tilde{f}$.  These data then give rise to a filtered Floer-Novikov complex as follows.  The finite set $S$ is the zero-set $\{p_1,\ldots,p_n\}$ of $\theta$, while the principal $\Gamma$-bundle $P\to S$ is just the restriction $\pi|_{\pi^{-1}(S)}$.  The action functional $\mathcal{A}\co P\to\mathbb{R}$ is given by $\mathcal{A}=\tilde{f}|_{P}$.  The period map $\omega$ is given by, for $g\in \Gamma=\pi_1(M)/\pi_1(\tilde{M})$, setting $\omega(g)=-\int_{\gamma}\theta$ for $\gamma$ an arbitrary loop representing $g$.   If  $\tilde{p}_i,\tilde{p}_j\in P$, the numbers $n(\tilde{p}_i,\tilde{p}_j)$ are zero unless $ind_{\tilde{f}}(\tilde{p}_i)=ind_{\tilde{f}}(\tilde{p}_j)+1$ (where $ind_{\tilde{f}}$ denotes the Morse index), in which case $n(\tilde{p}_i,\tilde{p}_j)$ is  obtained by counting integral curves $\gamma\co \mathbb{R}\to \tilde{M}$ of the negative gradient vector field of $\tilde{f}$ with respect to the pullback of a generic metric on $M$, where we require $\gamma(t)\to\tilde{p}_i$ as $t\to -\infty$ and $\gamma(t)\to \tilde{p}_j$ as $t\to\infty$.

At least when $R$ is a field, the resulting Novikov chain complex $C_*(\mathfrak{c})$ is chain homotopy equivalent to to $C_*(\tilde{M})\otimes_{R[\Gamma]}\Lambda_{\Gamma,\omega}$ (see \cite{F}).  Write $i_*\co H_*(\tilde{M})\to H_*(\mathfrak{c})$ for the map induced by coefficient extension by $\Lambda_{\Gamma,\omega}$.   
  For a class $i_*a\in H_*(\mathfrak{c})$ to satisfy $\rho(i_*a)=-\infty$ is closely analogous to the concept of $a\in C_*(\tilde{M})$ being \emph{movable to infinity} in the sense of \cite{FS}, since that $\rho(i_*a)=-\infty$ means that $a$ can be obtained from critical points of the action functional on $\tilde{M}$ (this functional being a primitive for $\pi^*\theta$) having arbitrarily large negative action.  In \cite{FS} it is shown that if $[\theta]\in H^*(M;\mathbb{R})$ has rank equal to the rank of $\Gamma$ then a class $a\in H_*(\tilde{M})$  can be moved to infinity if and only if $i_*a=0\in C_*(\tilde{M})\otimes_{R[\Gamma]}\Lambda_{\Gamma,\omega}$, consistently with Theorem \ref{nontriv}.

We note that since Theorems \ref{nontriv} and \ref{tight} do not require any injectivity hypothesis on the map $\omega$, we can take for $\tilde{M}$ above any abelian cover $\pi\co \tilde{M}\to M$ such that $\pi^*\theta$ is exact; thus our theorems are valid for Novikov homology with arbitrary abelian local coefficient systems.

\subsection{Hamiltonian Floer homology} Let $(M,\omega)$ be a symplectic manifold and let $H\co S^1\times M\to\mathbb{R}$ be a smooth function (we identify $S^1=\mathbb{R}/\mathbb{Z}$).  Let $X_H$ be the time-dependent vector field defined by $d(H(t,\cdot))=\iota_{X_H}\omega$, and let $\phi_H\co M\to M$ be the time-1 flow of $X_H$.  Let $\mathcal{L}_0M$ denote the space of contractible  loops $\gamma\co S^1\to M$.  Assuming that $H$ is nondegenerate in the sense that the  graph of $\phi_H$ is transverse to the diagonal of $M\times M$, the set \[ S=\{\gamma\in \mathcal{L}_0M| \dot{\gamma}(t)=X_H(t,\gamma(t))\}\] is finite.  Define \[ \widetilde{\mathcal{L}_0M}=\frac{\{(\gamma,w)\in \mathcal{L}_0M\times Map(D^2,M)| w|_{\partial D^2}=\gamma\}}{(\gamma,w)\sim (\gamma',w') \mbox{ if }\gamma=\gamma', \int_{D^2}w^*\omega=\int_{D^2}w'^*\omega,\mbox{ and }\langle c_1(M),[w'\#\bar{w}]\rangle=0}.\]   
The projection $\widetilde{\mathcal{L}_0M}\to\mathcal{L}_0M$ then restricts over the finite set $S$ to a principal $\Gamma$-bundle $P\to S$, where \[ \Gamma=\frac{\pi_2(M)}{\ker(\langle c_1,\cdot\rangle)\cap \ker(\langle [\omega],\cdot\rangle)}.\]  (Here, \emph{e.g.}, $\langle c_1,\cdot\rangle$ denotes the map defined on $\pi_2(M)$ by composing the Hurewicz map with evaluation of $c_1(M)\in H^2(M;\mathbb{Z})$).  The period map $\omega\co\Gamma\to\mathbb{R}$ is given by $\langle [\omega],\cdot\rangle$.  The action functional $\mathcal{A}\co P\to \mathbb{R}$ is given by \[ \mathcal{A}([\gamma,w])=-\int_{D^2}w^*\omega-\int_{0}^{1}H(t,\gamma(t))dt.\]    Finally, the numbers $n([\gamma,w],[\gamma',w'])$ are obtained by enumerating rigid solutions $u\co \mathbb{R}\times S^1\to M$ to the perturbed Cauchy-Riemann equation \[ \frac{\partial u}{\partial s}+J(t,u(s,t))\left(\frac{\partial u}{\partial t}-X_H(t,u(s,t))\right)=0\] which satisfy $u(s,\cdot)\to\gamma$ as $s\to\-\infty$, $u(s,\cdot)\to\gamma'$ as $s\to +\infty$, and $w\#u=w'$.  Here $J(t,\cdot)$ is a generic family of $\omega$-compatible almost complex structures on $TM$.  See \cite{Sal} for a survey of the details of the construction for a large family of symplectic manifolds, and \cite{FO},\cite{LT} for the general case.  (Conventionally, $R$ is usually taken to be either $\mathbb{Z}_2$, $\mathbb{Z}$, or $\mathbb{Q}$;  when the virtual cycle methods of \cite{FO} and \cite{LT} are needed, it is necessary to take  $R$ to be a field of characteristic zero.) 
A crucial property of the resulting Floer homology is that it is canonically isomorphic to the quantum homology of $(M,\omega)$.  

As mentioned earlier, Hamiltonian Floer homology is the Floer theory for which the spectral numbers have been most heavily developed, beginning with Schwarz's work \cite{S} and continuing with papers by Oh such as \cite{Oh1}, \cite{Oh2} (in turn, Schwarz's work was motivated in part by earlier work of Viterbo and Oh on Lagrangian submanifolds).  One of the earlier properties to be established was a nontriviality property analagous to Theorem \ref{nontriv} above, which follows as a result of the nature of the isomorphism with quantum homology.  The analogue of Theorem \ref{tight}, on the other hand, has been more of a challenge.
Theorem \ref{tight} in particular implies that, for any nonzero $a\in H_*(\mathfrak{c})$, we have
 \[ \rho(a)\in Im(\mathcal{A}\co P\to \mathbb{R}) \]  

The set $Im(\mathcal{A}\co P\to\mathbb{R})$ is known in the literature as the action spectrum $Spec(H)$ of $H$, and the fact that $\rho$ takes its values there is known as the spectrality axiom for $\rho$.  We accordingly emphasize that we have proven:

\begin{cor} Let $H$ be a nondegenerate Hamiltonian on any closed symplectic manifold.  Then the spectral number $\rho$ of the Floer homology of $H$ satisfies the spectrality axiom.\end{cor}

The main results are consequences of a result (Theorem \ref{alg}) about homomorphisms of free finite-rank modules over Novikov rings such as $\Lambda_{\Gamma,\omega}$.  The next section is devoted to stating and proving that result, while in the final section we will deduce Theorems \ref{nontriv} and \ref{tight} from Theorem \ref{alg}.

\subsection*{Acknowledgement}
I am grateful to Y.-G. Oh for his comments on earlier drafts of this paper.

\section{Approximation over Novikov rings}

Throughout this section $K$ will denote a  ring (as explained later, in practice $K$ will be a quotient of a polynomial ring over the ring $R$ that appeared in the last section, which is why we use a different notation for it), and $G\leq \mathbb{R}$ will denote an additive subgroup of $\mathbb{R}$.  Except in this section's closing Remark \ref{non-noeth}, we will always \textbf{assume that $K$ is Noetherian}. The Novikov ring of $G$ over $K$ is then, by definition, \[ \Lambda_K(G)=\{\sum_{g\in G}c_gT^g|c_g\in K,(\forall C>0)(\#\{g|c_g\neq 0,g<C\}<\infty)\}.\] For $a=\sum c_gT^g\in \Lambda_K(G)$ define $\nu(a)=\min\{g|c_g\neq 0\}$ (so $\nu(a)=\infty$ if and only if $a=0$).  For any $n$, this induces a function \begin{align*} \barnu\co \Lambda_K(G)^n&\to\mathbb{R}\cup\{\infty\} \\(a_1,\ldots,a_n)&\mapsto \min_{1\leq i\leq n}\nu(a_i)\end{align*} which attains the value $\infty$ only at the zero vector.  Note that $\barnu$ satisfies the non-Archimedean triangle inequality \[ \barnu(v+w)\geq \min\{\barnu(v),\barnu(w)\},\] which is in fact an equality whenever $\barnu(w)\neq\barnu(v)$.  

Setting $d(v,w)=e^{-\barnu(v-w)}$ then makes $\Lambda_K(G)^n$ into a non-Archimedean metric space.  The goal of this section is to show that if $V\leq \Lambda_K(G)^n$ is a submodule over $\Lambda_K(G)$, then  any $w\in \Lambda_K(G)^n$ has a best approximation in $V$ with respect to the metric $d$, and also with respect to the metrics obtained by replacing $\barnu$ with certain other functions $\barnu_{\vec{t}}$ in the formula for $d$.  Note that if $G\leq \mathbb{R}$ is a dense subgroup and $K$ is a field (which implies that $\Lambda_K(G)$ is a non-Archimedean field) the example given in Section 3 of \cite{IH} can be adapted to give a non-Archimedean norm on $\Lambda_K(G)^2$ with respect to which $(0,1)$ does not have a best approximation in $\Lambda_K(G)\times\{0\}$, so the property which we are to prove depends in a meaningful way on the functions $\barnu_{\vec{t}}$ and is not just a consequence of $\Lambda_K(G)^n$ having finite rank.

Define \[ \Lambda_K(G)_{\geq 0}=\{a\in \Lambda_K(G)|\nu(a)\geq 0\},\quad \Lambda_K(G)_{+}=\{a\in \Lambda_K(G)|\nu(a)> 0\},\] and similarly, for any $\Lambda_K(G)$-submodule $V$ of $\Lambda_K(G)^n$, \[
V_{\geq 0}=\{v\in V|\barnu(v)\geq 0\},\quad V_+=\{v\in V|\barnu(v)>0\}.\]

Note that $V_{\geq 0}$ is a $\Lambda_K(G)_{\geq 0}$-module.

Our argument will twice make use of the following lemma.

\begin{lemma}\label{fixedpt}  Let $\{u_1,\ldots,u_k\}\in \Lambda_K(G)^n$ with $\barnu(u_i)=0$, and let \[ U=span_{\Lambda_K(G)}\{u_1,\ldots,u_k\}.\] Let $V\leq \Lambda_K(G)^n$ be any $\Lambda_K(G)$-submodule such that $U\leq V$.  Suppose that $\phi\co V\to V$ is any function with the following properties: \begin{itemize} \item[(i)] For all $v\in V$, either $\phi(v)=v$ or $\barnu(\phi(v))>\barnu(v)$ (so in particular $\phi(0)=0$).
\item[(ii)] If $\phi(v)\neq v$, then $v-\phi(v)\in span_K\{T^{\barnu(v)}u_1,\ldots,T^{\barnu(v)}u_k\}$.
\end{itemize}
Then for every $v\in V$ there is $u\in U$ such that \[ \phi(v-u)=v-u\] and either $u=0$ or else \[ \barnu(u)=\barnu(v),\, T^{-\barnu(u)}u\in span_{\Lambda_K(G)_{\geq 0}}\{u_1,\ldots,u_k\}.  \] 
\end{lemma}
\begin{remark} Note that $\phi$ need  not be an additive group homomorphism (much less a module homomorphism).\end{remark}
\begin{proof} Now any $w\in \Lambda_K(G)^n$ can be expressed in the form $w=\sum_g w_gT^g$ where $w_g\in K^n$,  and we have $\barnu(w)=\min\{g|w_g\neq 0\}$.  Given any finite subset $S\subset \Lambda_K(G)^n$, define \[ N(S)=\{g|(\exists w\in S)(w_g\neq 0)\}.\]  Note that the finiteness of $S$ and the definition of the Novikov ring show that $N(S)$ is always discrete and bounded below.  Where $\phi$ is as in the statement of the lemma, for any $v\in V$ we see that (since $\phi(v)$ differs from $v$ by an element of the span \textbf{over \textit{K}} of the $T^{\nu(v)}u_i$, and since $0\in N(\{u_1,\ldots,u_k\})$) we have \begin{equation}\label{Ninc} N(\{\phi(v),v-\phi(v)\})\subset N(\{v\})+N(\{u_1,\ldots,u_k\}),\end{equation} where we use the usual notation $A+B=\{a+b|a\in A,b\in B\}$ for sets $A,B\subset\mathbb{R}$.

Define sequences $\{v^{(j)}\}_{j=0}^{\infty}$, $\{w^{(j)}\}_{j=0}^{\infty}$ of elements of, respectively, $V$ and $U$, by \[ v^{(0)}=v, \,\, w^{(0)}=0,\,\, v^{(j+1)}=\phi(v^{(j)}),\,\, w^{(j+1)}=v^{(j)}-\phi(v^{(j)}).\]

By (\ref{Ninc}) and induction on $j$, we see that for any $j$ \[ N(\{v^{(j)},w^{(j)}\})\subset N(\{v\})+j(N(\{u_1,\ldots,u_k\}))\subset N(\{v\})+\cup_{r=1}^{\infty}r(N(\{u_1,\ldots,u_k\})),\] where for $A\subset\mathbb{R}$ and $j\in\mathbb{N}$ we define $j(A)=\{\sum_{i=1}^{j}a_i|a_i\in A\}$.  Now since $\barnu(u_i)= 0$, the set $N(\{u_1,\ldots,u_k\})$ is a discrete set of nonnegative numbers; hence the set \[\cup_{r=1}^{\infty}r(N(\{u_1,\ldots,u_k\})),\] which consists of nonnegative-integer linear combinations of elements of 
$N(\{u_1,\ldots,u_k\})$, is also a discrete set of nonnegative numbers.  So since $N(\{v\})$ is discrete and bounded below, it follows that the set $Z=N(\{v\})+\cup_{r=1}^{\infty}r(N(\{u_1,\ldots,u_k\}))$ is discrete as well.

Now $\barnu(v^{(j)})$ is a monotone increasing sequence in this discrete set $Z$, and if for some $N$ we have $\barnu(v^{(N+1)})=\barnu(v^{(N)})$ then $v^{(j)}=v^{(N)}$ for all $j\geq N$, so either there is some minimal $N$ such $\barnu(v^{(j)})=\barnu(v^{(N)})$ for all $j\geq N$, or else $\barnu(v^{(j)})\to\infty$.

In the first case, by the defining properties of $\phi$ we see that $\phi(v^{(N)})=v^{(N)}$, and \[ v-v^{(N)}=\sum_{j=1}^{N}(v^{(j-1)}-v^{(j)})=\sum_{j=1}^{N}w^{(j)}.\]  Each $w^{(j)}$ belongs to $T^{\barnu(v^{(j-1)})}span_K\{u_1,\ldots,u_k\}$, so since the $\barnu(v^{(j)})$ form a monotone increasing sequence beginning at $\barnu(v)$ it follows that, where $u=\sum_{j=1}^{N}w^{(j)}$, we have  \[ T^{-\barnu(v)}u\in span_{\Lambda_K(G)_{\geq 0}}\{u_1,\ldots,u_k\}.\] Furthermore, unless $v^{(N)}=v$ (\emph{i.e.}, unless $u=0$) one has $\barnu(v^{(N)})=\barnu(v-u)>\barnu(v)$, which forces $\barnu(u)=\barnu(v)$.  Thus $u$ is as required.

There remains the case that $\barnu(v^{(j)})\to\infty$.    
  Now $w^{(j)}=v^{(j-1)}-v^{(j)}$
can, by (ii), be written \[ w^{(j)}=\sum_i a_{ij}u_iT^{\barnu(v^{(j-1)})}\quad (a_{ij}\in K),\] so, using that the $\barnu(v^{(j)})$ strictly increase from $\barnu(v)$ and diverge to $\infty$, \[ u=\sum_{j=1}^{\infty}w^{(j)}=\sum_{i=1}^{k}\left(\sum_{j=1}^{\infty}a_{ij}T^{\barnu(v^{(j-1)})}\right)u_i\]
validly defines an element of $T^{\barnu(v)}span_{\Lambda_K(G)_{\geq 0}}\{u_1,\ldots,u_k\}$.  For any $N$, one has (using the fact that the $\barnu(w^{(j)})=\barnu(v^{(j-1)})$ are increasing in $j$) \[ \barnu(v-u)\geq \min\{\barnu(v-\sum_{j=1}^{N}w^{(j)}),\barnu(\sum_{j=N+1}^{\infty}w^{(j)})\}=\min\{\barnu(v^{(N)}),\barnu(w^{(N+1)})\}=\barnu(v^{(N)});\] that this holds for all $N$ forces $\barnu(v-u)=\infty$, \emph{i.e.}, $v=u$, so the required properties of $u$ follow immediately.  
\end{proof}

For a submodule $V\leq \Lambda_K(G)^n$, define \[ \tilde{V}=V_{\geq 0}/V_+.\]  Note that one has $K\cong \Lambda_K(G)_{\geq 0}/\Lambda_K(G)_+$, with the quotient projection corresponding by this isomorphism to $\sum_g a_gT^g\mapsto a_0$.  $\tilde{V}$ is then a $K$-module, and is a submodule of $\widetilde{\Lambda_K(G)^n}\cong K^n$.  In particular, since $K$ is Noetherian (and submodules of finitely generated modules over Noetherian rings are finitely generated), $\tilde{V}$ is finitely generated over $K$.

For $v\in V_{\geq 0}$, let $\tilde{v}\in\tilde{V}$ denote the image of $v$ under the quotient map $V_{\geq 0}\to \tilde{V}$.  Our first application of Lemma \ref{fixedpt} is the following, which in retrospect is analogous to Lemma A.4.11 in \cite{FOOO}.

\begin{lemma} \label{span}If $V\leq \Lambda_K(G)^n$ and $u_1,\ldots,u_k\in V_{\geq 0}$ are such that $\tilde{V}=span_K\{\tilde{u}_1,\ldots,\tilde{u}_k\}$, then \[ V_{\geq 0}=span_{\Lambda_K(G)_{\geq 0}}\{u_1,\ldots,u_k\}.\]
\end{lemma}

\begin{proof} Since if for some $i$ we had $\barnu(u_i)\neq 0$ then $\tilde{u}_i$ would vanish in $\tilde{V}$, by removing $u_i$ if necessary we may as well assume that each $\barnu(u_i)=0$.  We define a function $\phi\co V\to V$ as follows.  First set $\phi(0)=0$.  If $v\in V$ is nonzero, we have $\barnu(T^{-\barnu(v)}v)=0$.  Since the $\tilde{u}_i$ span $\tilde{V}$ over $K$, we can pick $x_1(v),\ldots,x_k(v)\in K$ such that \[ \widetilde{T^{-\barnu(v)}v}=\sum_{i=1}^{k}x_i(v)\tilde{u}_i\in \tilde{V}.\]  Thus \[ \barnu(T^{-\barnu(v)}v-\sum_{i=1}^{k}x_i(v)u_i)>0.\]  So if we set $\phi(v)=v-T^{\barnu(v)}\sum_{i=1}^{k}x_i(v)u_i$ for $v\neq 0$, $\phi$ now satisfies the hypotheses of Lemma \ref{fixedpt} together with the additional property that its only fixed point is $0$.  This latter property then forces the $u$ that is found by the lemma for any given $v$ to be equal to $v$.  Thus for any nonzero $v\in V$, we have \[ T^{-\barnu(v)}v\in span_{\Lambda_K(G)_{\geq 0}}\{u_1,\ldots,u_k\};\] in particular, if $v\in V_{\geq 0}$ then \[ v\in span_{\Lambda_K(G)_{\geq 0}}\{u_1,\ldots,u_k\}.\]
\end{proof}

\begin{lemma}\label{tzero} Let $U\leq \Lambda_K(G)^n$ be any submodule, and let $w\in \Lambda_K(G)^n$.  Then there is $u\in U$ such that \[ \barnu(w-u)=\sup_{v\in U}\barnu(w-v)\mbox{ and either }u=0 \mbox{ or } \barnu(u)=\barnu(w).\]
\end{lemma}

\begin{proof} As noted earlier, since $\tilde{U}$ is a submodule of the finitely generated module $\widetilde{\Lambda_K(G)^n}\cong K^n$ over the Noetherian ring $K$, there are $u_1,\ldots,u_k\in U$ with $\barnu(u_i)=0$ such that $\tilde{u}_1,\ldots,\tilde{u}_k$ span $\tilde{U}$ over $K$.  So by Lemma \ref{span}, $u_1,\ldots,u_k$ span $U_{\geq 0}$ over $\Lambda_K(G)_{\geq 0}$ (from which it of course follows that they span $U$ over $\Lambda_K(G)$ since any element $v\in U$ satisfies $T^gv\in U_{\geq 0}$ for suitable $g$).  Define a function $\phi\co \Lambda_K(G)^n\to\Lambda_K(G)^n$ as follows.  First, if there is no $u\in U$ with the property that $\barnu(w-u)>\barnu(w)$, set $\phi(w)=w$  (so in particular $\phi(0)=0$).  Suppose now that $w$ is such that such a $u$ does in fact exist. 
 This supposition then amounts to the statement that there is $v(w)\in U$ such that \begin{equation}\label{reducew} \barnu(T^{-\barnu(w)}w-v(w))>0.\end{equation}  Now since $\barnu(T^{-\barnu(w)}w)=0$ this forces $\barnu(v(w))\geq 0$, \emph{i.e.}, $v(w)\in U_{\geq 0}$.  So since the $u_i$ span $U_{\geq 0}$ over $\Lambda_K(G)_{\geq 0}$ there are $y_i=\sum_g y_{i,g}T^g\in \Lambda_K(G)_{\geq 0}$ (where $y_{i,g}\in K$) such that $v(w)=\sum_{i=1}^{k}y_iu_i$.  Now set $v'(w)=\sum_{i=1}^{k}y_{i,0}u_i$.  We then have $\barnu(v(w)-v'(w))>0$, which together with (\ref{reducew}) implies that $\barnu(T^{-\barnu(w)}w-v'(w))>0$.  So set $\phi(w)=w-T^{\barnu(w)}v'(w)$.  This completes the definition of $\phi$; since $v'(w)$ is a $K$-linear combination of the $u_i$ $\phi$ satisfies the hypotheses of Lemma \ref{fixedpt}.  Hence for any $w$ there is $u\in U$ such that $\phi(w-u)=w-u$ and either $u=0$ or $\barnu(u)=\barnu(w)$.  Now $\phi$ was defined in such a way that the only fixed points of $\phi$ are those $w'$ such that $\barnu(w')=\sup_{v\in U}\barnu(w'-v)$.  Hence \[ \barnu(w-u)=\sup_{v\in U}\barnu(w-u-v)=\sup_{v\in U}\barnu(w-v),\] as desired.
\end{proof}

\begin{theorem} \label{alg}  Let $A\in M_{n\times m}(\Lambda_K(G))$ and $\vec{t}=(t_1,\ldots,t_n)\in \mathbb{R}^n$.  For $(a_1,\ldots,a_n)\in \Lambda_K(G)^n $
define \[ \barnu_{\vec{t}}(a_1,\ldots,a_n)=\min_{1\leq i\leq n}(\nu(a_i)-t_i).\]  Then there is $\gamma\in\mathbb{R}$, depending only on $\vec{t},A,$ with the following property.  If $w\in \Lambda_K(G)^n$ then there is $x_0\in \Lambda_K(G)^m$ such that \[ \barnu(x_0)\geq \barnu_{\vec{t}}(w)-\gamma\mbox{  and  } \barnu_{\vec{t}}(w-Ax_0)=\sup_{x\in \Lambda_K(G)^m}\barnu_{\vec{t}}(w-Ax).\]
\end{theorem}

\begin{proof} If $\vec{t}=\vec{0}$, since we have $\barnu_{\vec{0}}=\barnu$, except for the existence of $\gamma$ this is just Lemma \ref{tzero} applied to the submodule $U=A(\Lambda_K(G)^m)$ of $\Lambda_K(G)^n$.  To obtain $\gamma$, as in the proof of Lemma \ref{tzero} let $u_1,\ldots,u_k$ generate $A(\Lambda_K(G)^m)_{\geq 0}$ over $\Lambda_K(G)_{\geq 0}$.  Let $x_1,\ldots,x_k\in \Lambda_K(G)^m$ be such that $Ax_i=u_i$ ($1\leq i\leq k$), and set 
\[ -\gamma=\min_{1\leq i\leq k}\barnu(x_i).\]  Then if $u\in A(\Lambda_K(G)^m)$, so that $T^{-\barnu(u)}u\in A(\Lambda_K(G)^m)_{\geq 0}$, letting $a_i\in \Lambda_K(G)_{\geq 0}$ with $T^{-\barnu(u)}u=\sum a_iu_i$ we have $u=Ax$ where $x=\sum_{i}T^{\barnu(u)}a_ix_i$ satisfies $\barnu(x)\geq \barnu(u)-\gamma$.  So given $w\in \Lambda_K(G)^n$, if $u\in A(\Lambda_K(G)^m)$ is as in the conclusion of Lemma \ref{tzero} then $x_0=x$ as constructed in the previous sentence will have the desired properties.

We now deduce the theorem for general $\vec{t}\in\mathbb{R}^n$ from the already-proven special case that $\vec{t}=0$.

Let $G'$ be any additive subgroup of $\mathbb{R}$ that contains both $G$ and $\{t_1,\ldots,t_n\}$.

Where $\hat{e}_1,\ldots,\hat{e}_n$ is the standard basis for $\Lambda_K(G')^n$, consider the basis $\hat{e}'_1,\ldots,\hat{e}'_n$ for $\Lambda_K(G')^n$ given by $\hat{e}'_i=T^{t_i}\hat{e}_i$.  Viewing $A\in M_{n\times m}(\Lambda_K(G))$ as a matrix with coefficients in the larger Novikov ring $\Lambda_K(G')$, the matrix representing the underlying homomorphism of $A$ with respect to the standard basis for $\Lambda_K(G')^m$ and the new basis $\hat{e}'_1,\ldots,\hat{e}'_n$ for $\Lambda_K(G')^n$ is $A'=MA$ where $M_{ij}=T^{-t_i}\delta_{ij}$.   Now  we have \[ \barnu_{\vec{t}}(\sum w'_i\hat{e}'_i)=\min_{1\leq i\leq n}(\nu(w'_iT^{t_i})-t_i)=\min_{1\leq i\leq n}\nu(w'_i)=\barnu(w'_1,\ldots,w'_n).\]  So for any $x=(x_1,\ldots,x_m)\in \Lambda_K(G')^m$,$w=(w_1,\dots,w_n)\in\Lambda_K(G')^n$ we have \[ \barnu_{\vec{t}}\left(w-Ax\right)=\barnu\left(T^{-t_1}w_1-(MAx)_1,\ldots,T^{-t_n}w_n-(MAx)_n\right).\]

Hence applying the $\vec{t}=\vec{0}$ case of the theorem  to the matrix $MA\in M_{n\times m}(\Lambda_K(G'))$ and the vector $(T^{-t_1}w_1,\ldots,T^{-t_n}w_n)\in \Lambda_K(G')^n$  shows that there is $x_0\in (\Lambda_K(G'))^m$ such that \[ \barnu_{\vec{t}}(w-Ax_0)=\sup_{x\in \Lambda_K(G')^m}\barnu_{\vec{t}}(w-Ax)\]
and $\barnu(x_0)\geq \barnu(T^{-t_1}w_1,\ldots,T^{-t_n}w_n)-\gamma=\barnu_{\vec{t}}(w)-\gamma$.

So all that remains is to show that if $w\in \Lambda_K(G)^n$ then this $x_0$ can be taken to lie in $\Lambda_K(G)^m$.  Now if $x\in \Lambda_K(G')^m$, each coordinate $x_i$ of $x$ has the form \[ x_i=\sum_{g\in G}a_{i,g}T^{g}+\sum_{g\in G'\setminus G}b_{i,g}T{^g};\] write $x'_i=\sum_{g\in G}a_{i,g}T^{g}$ and $x''_i=x_i-x'_i$.  Since $A$ has its coefficients in $\Lambda_K(G)$, each coordinate of  $Ax'$ belongs to $\Lambda_K(G)$, while each coordinate of $Ax''$ has the form $\sum_{g\in G'\setminus G}c_gT^g$.   So if $w\in \Lambda_K(G)^n$, no term in the expansion of any coordinate of $w-Ax'$ can cancel with a term in the expansion of any coordinate of $Ax''$ (for the former only involve exponents in $G$, while the latter only involve exponents in $G'\setminus G$).  In view of this, we have, for each $i$, $\nu((w-Ax')_i)\geq \nu((w-Ax)_i)$, and therefore \[ \barnu_{\vec{t}}(w-Ax')\geq \barnu_{\vec{t}}(w-Ax).\]

So where $x_0$ is as above, $x'_0\in \Lambda_K(G)^m$ will satisfy $\barnu(x'_0)\geq \barnu(x_0)\geq \barnu_{\vec{t}}(w)-K$ and 
\[ \barnu_{\vec{t}}(w-Ax'_0)\geq \barnu_{\vec{t}}(w-Ax_0)=\sup_{x\in \Lambda_K(G')^m}\barnu_{\vec{t}}(w-Ax),\] so since $\Lambda_K(G)^m\leq \Lambda_K(G')^m$ $x'_0$ fulfills the requirements of Theorem \ref{alg}.
\end{proof}
 
\begin{remark}[(A non-Noetherian counterexample)] \label{non-noeth}

We have assumed throughout this section that $K$ is a Noetherian ring; we present now a case in which Theorem \ref{alg} fails for a non-Noetherian $K$.  For some base field $k$, put 
\[ K=\frac{k[a_0,b_0,\ldots,a_n,b_n,\ldots]}{\langle \{a_m b_n|m-n\notin\{0,1\}\}\cup \{a_nb_n-a_0b_0|n\geq 0\}\rangle}.\]

(Incidentally, while this $K$ is not an integral domain, it is not difficult to modify the example we present here to a slightly more complicated one in which $K$ is an integral domain.)

For some additive subgroup $G\leq \mathbb{R}$, choose a sequence $\{\lambda_n\}_{n=1}^{\infty}$ in $G$ such that $\lambda_n\
\nearrow\infty$,  define \[ z=\sum_{n=0}^{\infty}a_nT^{\lambda_n}\in \Lambda_K(G),\] and consider the ideal $\langle z\rangle$ in $\Lambda_K(G)$ generated by $z$.   If \[ w=\sum_{n=0}^{\infty}(a_nT^{\lambda_n}+a_{n+1}T^{\lambda_{n+1}})b_n ,\] note that for any $N$ we have  \[ \left(\sum_{n=0}^{N}b_n\right)z=\sum_{n=0}^{N}(a_nT^{\lambda_n}+a_{n+1}T^{\lambda_{n+1}})b_n\] and so \[ \nu\left(w-\left(\sum_{n=0}^{N}b_n\right)x\right)=\lambda_{N+1}\to\infty \mbox{ as }N\to\infty.\]  However (recalling that $\sum_{n=0}^{\infty}b_n$ is not an element of $\Lambda_K(G)$) it is easily seen that $w\notin \langle z\rangle$.  Thus if $A\co\Lambda_K(G)\to\Lambda_K(G)$ is defined by $Ax=xz$ then the supremum of $\{\barnu(w-Ax)|x\in \Lambda_K(G)\}$ (namely $\infty$) is not attained by any $x$, and the analogue of Theorem \ref{alg} (with $m=n=1$) fails for this choice of $A$. 

One could also ask whether $\sup\{\barnu(w-Ax)|x\in \Lambda_K(G)\}$ could ever be finite and yet fail to be attained.  Of course if $G\leq \mathbb{R}$ is discrete, since $\barnu(w-Ax)$ always belongs to $G$ when it is finite, the supremum is indeed attained in this case.  However, suppose that $G$ is not discrete, and again let $z=\sum_{n=0}^{\infty}a_nT^{\lambda_n}$, with the $\lambda_n$ now chosen to have the property that $\lambda_{n+1}-\lambda_n\nearrow C$ for some finite $C>0$.  We then have \[ b_nT^{-\lambda_n}w=a_0b_0+a_{n+1}b_nT^{\lambda_{n+1}-\lambda_n},\] in view of which (again putting $Ax=xz$) $\sup\{a_0b_0-Ax|x\in\Lambda_K(G)\}$ is equal to $C$ and is not attained.
\end{remark}

\section{Proofs of the main theorems}

Let $\mathfrak{c}$ denote a filtered Floer-Novikov complex as defined in the introduction, with data $S,P,\Gamma,\mathcal{A},\omega,\partial$ as above, giving rise to the Floer complex $C_*(\mathfrak{c})$, which is a module over $\Lambda_{\Gamma,\omega}$. Recall also the function $\ell\co C_*(\mathfrak{c})\to \mathbb{R}\cup\{-\infty\}$, defined above. 

Let $\pi\co P\to S$ be the principal bundle projection in the definition of $\mathfrak{c}$.  Write $S=\{s_1,\ldots,s_n\}$ and choose and fix $p_i\in P$ such that $\pi(p_i)=s_i$.  By definition, then, we have \begin{align*}
C_*(\mathfrak{c})&=\left\{\sum_{i=1}^{n}\left(\sum_{g\in G} b_{g,i}g\right)p_i|b_{g,i}\in R,(\forall C\in\mathbb{R})(\#\{(i,g)|b_{g,i}\neq 0,\mathcal{A}(g\cdot p_i)>C\}<\infty)\right\} \\
& = \left\{\sum_{i=1}^{n}\left(\sum_{g\in G} b_{g,i}g\right)p_i|b_{g,i}\in R,(\forall i)(\forall C\in\mathbb{R})(\#\{g|b_{g,i}\neq 0, \omega(g)<C\}<\infty)\right\}, \end{align*} where we have used the formula $\mathcal{A}(g\cdot p_i)=\mathcal{A}(p_i)-\omega(g)$ from the definition of a filtered Floer-Novikov complex.

This provides us with an identification
\begin{equation}\label{cstar} C_*(\mathfrak{c})\cong\bigoplus_{i=1}^{n}\Lambda_{\Gamma,\omega}\langle p_i \rangle\cong \Lambda_{\Gamma,\omega}^n.\end{equation}  Note that, with respect to this identification, for $(a_1,\ldots,a_n)\in \Lambda_{\Gamma,\omega}^n$, we have \begin{equation}\label{ellform} \ell(a_1,\ldots,a_n)=\max_{1\leq i\leq n}(\mathcal{A}(p_i)-\omega(a_i)),\end{equation} where writing $a_i=\sum_ga_{i,g}g$ we set $\omega(a_i)=\min_{g:a_{i,g}\neq 0}\omega(g).$

We now turn attention to the Novikov ring $\Lambda_{\Gamma,\omega}$.  Since $\omega\co \Gamma\to\mathbb{R}$ is a homomorphism whose domain is a finitely generated abelian group and whose image $G\leq \mathbb{R}$ is torsion free,  the exact sequence $\ker\omega\rightarrowtail\Gamma\twoheadrightarrow  G$ splits and so identifies $\Gamma$ with $\ker\omega\oplus G$.  With respect to this identification, an element of $\Lambda_{\Gamma,\omega}$ is a formal sum of the type $\sum_{g\in G}\sum_{h\in \ker \omega}a_{g,h}s^hT^g$ ($a_{g,h}\in\mathbb{R}$) having the property that for each $C\in\mathbb{R}$ there are only finitely many nonzero $a_{g,h}$ with $g< C$.  This property holds if and only if both (i) For each $C>0$ there are only finitely many $g$ such that \emph{any} $a_{g,h}$ is nonzero and $g<C$, and (ii) For any $g$ there are just finitely many $h$ such that $a_{g,h}\neq 0$, so that, for any $g$, $\sum_h a_{g,h}s^h$ defines an element of the group ring $R[\ker\omega]$.  Thus, setting $K=R[\ker\omega]$, we have \[ \Lambda_{\Gamma,\omega}=\Lambda_K(G).\]  Moreover, setting $\vec{t}=(\mathcal{A}(p_1),\ldots,\mathcal{A}(p_n))$, (\ref{ellform}) gives \[ \ell(a_1,\ldots,a_n)=-\barnu_{\vec{t}}(a_1,\ldots,a_n).\]

Now since $\ker\omega$ is a finitely generated abelian group, $K=R[\ker\omega]$ is the quotient of a polynomial ring on finitely many variables over $R$, so by the Hilbert basis theorem and the fact that quotients of Noetherian rings are Noetherian, $K$ is Noetherian whenever $R$ is.

As such, Theorems \ref{nontriv} and \ref{tight} are now immediate consequences of Theorem \ref{alg}.  Namely, take $m=n$ in Theorem \ref{alg}, and take for $A$ the matrix representing the $\Lambda_K(G)$-module homomorphism $\partial$ with respect to the identification (\ref{cstar}).  By definition, if $a\in H_*(\mathfrak{c})$ and $c_0\in C_*(\mathfrak{c})$ is any representative of the class $a$ we have \[ \rho(a)=\inf\{\ell(c)|[c]=a\}=\inf\{\ell(c_0-\partial h)|h\in C_*(\mathfrak{c})\}=-\sup\{\barnu_{\vec{t}}(c_0-\partial h)|h\in C_*(\mathfrak{c})\}.\]  Theorem \ref{alg} then produces an $h$ attaining this infimum and such that \begin{equation}\label{ellh} -\ell(h)=\barnu_{\vec{t}}(h)\geq \barnu(h)-\max_i\mathcal{A}(p_i)\geq -\ell(c_0)-\gamma-\max_i\mathcal{A}(p_i).\end{equation}  $\alpha=c_0-\partial h$ is then a representative of $a$ satisfying $\rho(a)=\ell(\alpha)$, as required by Theorem \ref{tight}.  In particular, if $\rho(a)=-\infty$ we necessarily have $\ell(\alpha)=-\infty$, so $\partial h=c$, and (\ref{ellh}) gives $\ell(h)\leq \ell(c_0)+M$ where $M=\gamma+\max_i\mathcal{A}(p_i)$, proving Theorem \ref{nontriv}.

\end{document}